\documentclass[11pt]{amsart}
\usepackage{amsmath, amsthm, amssymb,amscd}
\usepackage{float}
\usepackage{graphicx}

\usepackage[all,cmtip]{xy}
\usepackage[english]{babel}

\newcommand{\PP}{\mathbb{P}}

\newcommand{\seq}{\subseteq}
\newcommand{\C}{\mathbb{C}}

\newtheorem{thm}{Theorem}[section]
\newtheorem*{thm-nl}{Theorem}
\newtheorem*{prop-nl}{Proposition}

\def\proj{{\textbf P}}

\def\H{\mathcal{H}}
\def\Pic0{{\rm Pic}^0(X)}

\newtheorem*{cor-nl}{Corollary}
\newtheorem{conjecture}[thm]{Conjecture}
\newtheorem*{conjecture-nl}{Conjecture}
\newtheorem*{quest-nl}{Question}
\newtheorem*{quests-nl}{Questions}

\newtheorem{prop}[thm]{Proposition}

\theoremstyle{remark}

\title{Syzygies of curves beyond Green's Conjecture}

\author{Michael Kemeny}

\begin{document}
\bibliographystyle{plain}

\maketitle

\begin{abstract}
We survey three results on syzygies of curves beyond Green's conjecture, with a particular emphasis on drawing connections between the study of syzygies and other topics in moduli theory.
\end{abstract}

\section{Introduction}
Arguably \emph{the} central concept of modern commutative algebra is that of a minimal free resolution. Let $S$ be either a local Noetherian ring or a positively graded, finitely generated algebra over a field. The task is then to describe the shape of the minimal free resolutions of various classes of $S$ modules. 

In this survey, we restrict attention to the graded ring
$$ R=\C[x_0, \ldots, x_m],$$
with grading $R_d=\{ \text{polynomials of degree $d$} \}.$
There is one particular class of graded $R$-modules which are the most interesting to algebraic geometers, namely those of the form
$$\Gamma_X(L):= \bigoplus_{n \geq 0} H^0(X,L^{\otimes n}),$$
where $X$ is a projective variety, $L$ is a very ample line bundle with $m$ sections, and where $\Gamma_X(L)$ is a $R\simeq \text{Sym}(H^0(X,L))$ module in the natural way. 

By the Hilbert Syzygy Theorem, any finitely generated, graded $R$ module $M$ has a minimal free resolution
$0 \leftarrow M \leftarrow F_0 \leftarrow F_1 \leftarrow \ldots \leftarrow F_k \leftarrow 0$
of length at most $k \leq m$. Decomposing each term 
$$F_i= \bigoplus_j R(-i-j)^{b_{i,j}(M)}$$
into its graded pieces, this resolution defines invariants $b_{i,j}(M)$ of the module, called the \emph{Betti numbers} of $M$.
\begin{quest-nl}
In the case $M=\Gamma_X(L)$, what geometric information about $X$ can be gleaned from the Betti numbers $b_{i,j}(X,L):=b_{i,j}(\Gamma_X(L))$?
\end{quest-nl}

The case which has received the most attention is when the variety is a projective curve $C$ of genus $g$ and $L=\omega_C$. For a projective curve, the geometric information one is most interested in is perhaps the \emph{gonality} $\text{gon}(C)$, which is the least degree $d$ such that there exists a degree $d$ cover $C \to \PP^1$ (we call such a cover a \emph{minimal pencil}). In practice, on sometimes needs to replace gonality with a slightly refined invariant, called the \emph{Clifford index} and defined as 
$$\text{Cliff}(C):=  \text{min} \; \{ \deg{M}-2h^0(M)+2 \; | \; M\in \text{Pic}(C), \; \deg(M) \leq g-1, \;  h^0(M) \geq 2 \}.$$
There is always the bound $\text{Cliff}(C) \leq \text{gon}(C)-2$, which, at least conjecturally, is an equality for most curves, \cite{ELMS}.
Here one has the celebrated:
\begin{conjecture-nl}[Green's Conjecture, \cite{green-koszul}]
We have the following vanishing of quadratic Betti numbers: $$b_{p,2}(C,\omega_C)=0 \; \;\text{for $p<\text{Cliff}(C)$}.$$
\end{conjecture-nl}
Here $\text{Cliff}(C)$ denotes the \emph{Clifford index} of the curve $C$. This conjecture provides a sweeping generalisation of the following classical results of Noether and Petri:
\begin{thm-nl} [Noether, Petri]
Assume $C$ is not hyperelliptic. Then the canonical embedding $C \hookrightarrow \PP^{g-1}$ is projectively normal. If, further, $C$ does not admit a degree three cover of $\PP^1$, then $I_{C/\PP^{g-1}}$ is generated by quadrics.
\end{thm-nl}

The first serious approach to Green's conjecture, due to Schreyer \cite{schreyer}, relies on the observation that if a curve admits a minimal pencil $f: C \to \PP^1$ of degree $d$, then the canonical curve $C$ lies on the rational normal scroll $X \seq \PP^{g-1}$ which can be geometrically described as the union of the span of the divisors $f^{-1}(p)$ in $\PP^{g-1}$ as $p \in \PP^1$ varies. The minimal free resolution of the scroll $X$ can be described by an \emph{Eagon--Northcott} resolution:
\begin{thm-nl}[Eagon--Northcott \cite{eagon-northcott}]
Let $R$ be a ring and $f: R^r \to R^s$ be a ring homomorphism for $r \geq s$. There is a complex
$$0 \leftarrow R \xleftarrow{\wedge^s f} \wedge^s R^r \leftarrow S_1^* \otimes \wedge^{s+1}R^r \leftarrow \ldots \leftarrow S_{r-s}^* \otimes \wedge^r R^r \leftarrow  0$$
where $S=R[x_1, \ldots, x_s]$, which is exact if and only if $\text{depth}\; (I_s(f))= r-s+1$.
\end{thm-nl}
\noindent In the theorem above, $I_{j}(f)$ denotes the ideal generated by $j \times j$ minors of $f$, for any $1 \leq j \leq s$.

By restriction to the canonical curve, the Eagon--Northcott resolution injects into the linear strand of the minimal resolution of the curve. This translates Green's Conjecture into the prediction that the length of the Eagon--Northcott resolution equals the length of the linear strand of the canonical curve.

 Eagon--Northcott resolutions form an important class of resolutions in their own right, and the work of Buchsbaum and Eisenbud on understanding and generalising these resolutions led to several results which have been seminal to the development of modern commutative algebra. For example, the famous \emph{Criterion for Exactness} came out of this:
 \begin{thm-nl}[Buchsbaum-Eisenbud \cite{buchsbaum-eisenbud-what}]
 Let $R$ be a ring and let
 $$F_{\bullet} \; : \; F_0 \xleftarrow{f_1} F_1 \xleftarrow{f_2} F_2 \leftarrow \ldots \xleftarrow{f_n} F_n \leftarrow 0$$
 be a complex of free $R$ modules.
 Assume \begin{enumerate}
 \item $\text{rank}\;F_i=\text{rank}\;f_i+\text{rank}\;f_{i+1}$
 \item if $I(f_i)\neq R$ then $\text{depth}\; I(f_i)\geq i$
 \end{enumerate}
 Then $F_{\bullet}$ is exact.
 \end{thm-nl}
 Here $I(f_i)$ is defined to be $I_{\text{rk}f_i}(f_i)$.
Using this criterion, Buchsbaum and Eisenbud construct resolutions generalising the Eagon--Northcott resolution in \cite{buchsbaum-eisenbud-generic}.

Perhaps surprisingly, some of the most effective tools for approaching Green's Conjecture have come from geometry rather than algebra. It was observed early on by Green and Lazarsfeld that there is an intimate connection between Green's Conjecture and the theory of K3 surfaces and moduli spaces of sheaves on such surfaces. This connection has proven to be surprisingly deep and has significantly influenced the subsequent development of K3 surface theory. As just one example, the following fundamental theorem came about as a verification of a prediction from Green's Conjecture:
\begin{thm-nl}[Green--Lazarsfeld \cite{green-lazarsfeld-special}] 
Let $X$ be a K3 surface and $L \in \text{Pic}(X)$ a base point free line bundle. Then $\text{Cliff}(C)$ is constant amongst all smooth curves $C \in |L|$.
 \end{thm-nl} In another direction, the study of syzygies of curves has proven to be remarkably important for the study of the birational geometry of the moduli space of curves, see \cite{farkas-popa}, \cite{farkas-syzygies}. 

In a landmark pair of papers, Green's Conjecture was eventually proven for a \emph{generic} curve of arbitrary genus by C.\ Voisin, \cite{voisin-even}, \cite{voisin-odd}. Voisin's proof relied on a new interpretation of the problem in terms of the Hilbert scheme of points on a K3 surface:
\begin{thm-nl}[Voisin]
Let $X$ be a complex projective variety and $L$ a line bundle. Then $b_{p,1}(X,L)$ equals the corank of the natural map
$$H^0(X^{[p+1]}_{\text{curv}},\text{det}\; L^{[p+1]}) \to H^0(\mathcal{I}_{p+1},(q^* \text{det}\; L^{[p+1]})_{|_{\mathcal{I}_{p+1}}})$$
where $X^{[p+1]}_{\text{curv}}$ is the curvilinear locus in the Hilbert scheme of points, $L^{[p+1]}$ is the tautological bundle, $\mathcal{I}_{p+1} \seq X^{[p+1]}_{\text{curv}}\times X$ is the incidence variety and $q: \mathcal{I}_{p+1}\to X^{[p+1]}_{\text{curv}}$ is the projection.
\end{thm-nl}

In this survey, we outline some new avenues of research going out in various directions from Green's Conjecture, with a focus on demonstrating the connections between the study of syzygies and other aspects of moduli theory. The first of these directions is the
Secant Conjecture of Green and Lazarsfeld,  \cite{green-lazarsfeld-projective}. This conjecture gives a condition for the vanishing of quadratic syzygies of a curve embedded by a nonspecial line bundle $L$. The Secant Conjecture generalises the following well-known theorem of Castelnuovo--Mumford in much the same way as Green's conjecture generalises the Theorem of Noether--Petri:
\begin{thm}[Castelnuovo--Mumford]
Let $L$ be a very ample line bundle on a curve with $ \deg(L) \geq 2g+1$ then $\phi_L: C \hookrightarrow \PP^r$ is projectively normal.
\end{thm} 
Green and Lazarsfeld proved that one can replace the bound in Castelnuovo--Mumford's theorem with $ \deg(L) \geq 2g+1-\text{Cliff}(C)$. The Secant Conjecture then extends the above results to higher syzygies. 

The second direction we discuss is the Prym--Green Conjecture. For a general canonical curve $C \hookrightarrow \proj^{g-1}$ the generic Green's Conjecture as proved by Voisin suffices to describe the shape of the free resolution of the homogeneous coordinate ring of $C$. The Prym--Green conjecture likewise predicts a similar shape for the homogeneous coordinate ring of a general \emph{paracanonical} curve, that is a curve embedded by a twist of the canonical line bundle by a torsion line bundle.

Lastly, we come back to Schreyer's original approach and consider the question of whether all the syzygies in the last position of the linear strand of a canonical curve come from the syzygies of the scrolls associated to the minimal pencils of the curve. Our approach to this question is a blend of Schreyer's original approach using the Eagon--Northcott complex with the approach of Hirschowitz--Ramanan, \cite{hirschowitz-ramanan}, which suggests that one should construct and study appropriate divisors on moduli spaces. In our case the moduli spaces are spaces of stable maps to $\proj^m$, which have some peculiarities in comparison with the moduli space of curves, and several natural questions are left open.

\noindent {\bf Acknowledgments:} It is a pleasure to thank the organisers of the Abel Symposium 2017 for a wonderful conference in a spectacular location. The results in this survey are joint work with my coauthor Gavril Farkas, who has taught me much of what I know about syzygies. I also thank D.\ Eisenbud and F.-O.\ Schreyer for enlightening conversations on these topics.

This survery is an amalgamation of material taken from my course on syzygies in Spring 2017 as well as talks given at UCLA and Berkeley in Autumn 2017. In particular, I thank Aaron Landesman for several corrections and improvements to my course notes.

\section{The Eagon--Northcott Complex}
It is quite rare for one to able to construct an explicit free resolution of a module, so those few families of resolutions which we have are much prized. Perhaps the most well known and useful resolutions is given by the \emph{Koszul complex}. Let $R$ be a ring and $\{f_1, \ldots, f_r \}$ a sequence of elements in $R$. Let $\bar{f}: R^r \to R$ be the map sending the $i^{th}$ basis element $e_i$ of $R^r$ to $f_i$ for $1 \leq i \leq r$. The Koszul complex associated to $\bar{f}$ is the complex
$$K_{\bullet}(\bar{f}) \; : \;  R \xleftarrow{\bar{f}} R^r \xleftarrow{d} \wedge^2 R^r \xleftarrow{d} \ldots \leftarrow \wedge^r R^r \leftarrow 0$$
where the differential is defined by
$$d(e_{i_1} \wedge \ldots \wedge e_{i_p})=\sum_{j=1}^p (-1)^{j+1} f_{i_j} \; e_{i_1}\wedge \ldots \wedge \hat{e}_{i_j} \wedge \ldots \wedge e_{i_p}.$$

The Eagon--Northcott complex was the very influential discovery that one could generalise the construction of the Koszul complex and associate a complex to any ring homomorphism $\bar{g}: R^r \to R^s$ with $r \geq s$. To describe how this works, we follow \cite{eisenbud-com} to construct the complex as a strand of a \emph{graded} Koszul complex. Consider the graded polynomial ring
$$S:=R[x_1,\ldots,x_s].$$ We may identify $S_1$ with $R^s$ via the standard basis. Setting
$$F:=S(-1)^{\oplus r} \simeq (R^r \otimes_R S)(-1),$$
then $\bar{g}$ defines a morphism $\tilde{g} :  F \to S$
of graded $S$ modules. Explicitly, if $e_1, \ldots, e_r$ is a basis of $R^r$, then $\widetilde{g}(e_i \otimes 1)=\bar{g}(e_i) \in R^s \simeq S_1$. Taking the Koszul complex of $\widetilde{g}$, we get
$$K_{\bullet}(\widetilde{g}) \; : \; S \leftarrow F \leftarrow \wedge^2 F \leftarrow \ldots \leftarrow \wedge^r F \leftarrow 0.$$
This is a \emph{graded} complex and taking the $k^{th}$ graded piece of this complex yields 
$$K_{\bullet}(\widetilde{g})_k \; : \;  S_k \xleftarrow{\delta} S_{k-1} \otimes_R R^r \xleftarrow{\delta} S_{k-2} \otimes \wedge^2 R^r \leftarrow \ldots \leftarrow S_{k-r} \otimes \wedge^r R^r \leftarrow 0.$$
Dualizing this and using the identification $\wedge^i R^r \simeq \wedge^{r-i}(R^r)^*$ we get a complex
$K^*_{\bullet}(\widetilde{g})_k \; : \;  S^*_{k-r} \xleftarrow{d} S^*_{k-r+1} \otimes R^r \leftarrow \ldots \leftarrow S^*_{k} \otimes \wedge^r R^r \leftarrow 0.$
Now set $k=r-s$. The first $s$ terms $S^*_{-s},\ldots, S^*_{-1}$ are all zero, so we get a complex
$$\wedge^s R^r \xleftarrow{d} S^*_1 \otimes \wedge^{s+1} R^r \leftarrow \ldots \leftarrow S^*_{r-s}\otimes \wedge^rR^r \leftarrow 0.$$
The Eagon--Northcott complex 
$$EN(\bar{g}) \; : \;  R \xleftarrow{\wedge^s \bar{g}} \wedge^s R^r \xleftarrow{d} S^*_1 \otimes \wedge^{s+1} R^r \leftarrow \ldots \leftarrow S^*_{r-s}\otimes \wedge^rR^r \leftarrow 0,$$
is then obtained by extending the above complex by $\bar{g}$ (it remains a complex).

Let $I_j(\bar{g})$ denote the ideal generated by the $j \times j$ minors of $\bar{g}$. 
\begin{thm-nl}[\cite{eagon-northcott}]
The complex $EN(\bar{g})$ is exact if and only if $\text{depth} \; I_s(\bar{g}) =r-s+1$.
\end{thm-nl}
For example, in the special case $s=1$, this gives us the well-known statement that $K_{\bullet}(\bar{f})$ is exact if and only if $\{f_1, \ldots, f_r \}$ forms a regular sequence.

Let us now restrict attention to the graded polynomial ring $R=\C[x_0, \ldots, x_m]$ and set $V=R_1$, which is an $m+1$ dimensional complex vector space. Then $\{ x_0, \ldots, x_m \}$ forms a regular sequence, and and taking the Koszul complex produces the resolution 
$$ 0 \leftarrow \C \leftarrow R \leftarrow V \otimes_{\C} R \leftarrow \wedge^2 V \otimes R \leftarrow \ldots \leftarrow \wedge^{m+1} V \otimes R \leftarrow 0,$$
of $\C \simeq R/ (x_0,\ldots, x_m)$. Let $M$ be a graded $R$ module, and consider the minimal free resolution $$0 \leftarrow M \leftarrow F_0 \leftarrow F_1 \leftarrow \ldots \leftarrow F_k \leftarrow 0.$$
The \emph{Betti numbers} of $b_{i,j}(M)$ are defined to be the $(i+j)^{th}$ graded piece of $\text{Tor}_R^i$ $$b_{i,j}(M):= \dim \text{Tor}_R^i(M,\C)_{i+j}.$$ In terms of the minimal free resolution above, $F_i= \bigoplus_j R(-i-j)^{b_{i,j}(M)}$. By tensoring the Koszul resolution of $\C$ by $M$ and using symmetry of $\text{Tor}$, one immediately obtains a more concrete description of the syzygy spaces $\text{Tor}_R^i(M,\C)_{i+j}$:
\begin{prop} \label{koszul-coh}
The $(i,j)^{th}$ syzygy space $\text{Tor}_R^i(M,\C)_{i+j}$ is the middle cohomology of 
$$\bigwedge^{i+1}V \otimes M_{j-1} \to  \bigwedge^i V \otimes M_j \to \bigwedge^{i-1}V \otimes M_{j+1}, $$
where the maps are the Koszul differentials and $V=R_1$.
\end{prop}

\section{Rational Normal Scrolls and Schreyer's Approach} \label{Schreyer}
In this section we discuss Schreyer's approach to Green's conjecture using rational normal scrolls.
Let $C$ be a smooth, projective curve of genus $g$, and let
$$R:=\text{Sym}\; H^0(C,\omega_C) \simeq \C[x_0,\ldots,x_{g-1}].$$
We are interested in the Betti numbers of the graded $R$ module
$$\Gamma_{C}(\omega_C):=\bigoplus_{n \geq 0} H^0(C,\omega^{\otimes n}).$$
We define $$b_{i,j}(C,\omega_C):=b_{i,j}(\Gamma_{C}(\omega_C)).$$
Suppose $C$ has gonality $d$ and let $f  : C \to \PP^1$ be a minimal pencil, i.e.\ suppose $\phi$ has the minimal degree $d$. It was observed by Schreyer \cite{schreyer} and Green--Lazarsfeld \cite[Appendix]{green-koszul}, that the minimal pencil imposes conditions on the possible values of the Betti numbers $b_{p,1}(C,\omega_C)$. In this section, we describe Schreyer's approach.

Suppose $C$ is not hyperelliptic, so that $|\omega_C|$ embeds the curve $C$ in $\proj^{g-1}$. If $f: C \to \PP^1$ is a minimal pencil, then for a general $p \in \PP^1$, the divisor $f^{-1}(p)$ is a finite set of $d$ points in $\PP^{g-1}$.
By geometric Riemann--Roch, the dimension of the span $<f^{-1}(p)>$ is given by
$$ \dim<f^{-1}(p)>=d-h^0(C,f^{*}\mathcal{O}_{\PP^1}(1))=d-2. $$
Now consider the union
$$X_f:=\bigcup_{p \in \PP^1} <f^{-1}(p)>.$$
Then $X_f \seq \PP^{g-1}$ is smooth, $d-1$ dimensional projective variety known as a \emph{rational normal scroll} which furthermore contains the curve $C \seq \PP^{g-1}$. Rational normal scrolls have the minimal possible degree
$$\deg(X)=1+\text{codim}X,$$
and as such they have been widely studied since Bertini classified varieties of minimal degree in 1907, \cite{bertini}.

 Following \cite{eisenbud-harris-centennial}, one may give a determinantal description of the variety $X_f$. Let $u,v$ be a basis of $H^0(C,f^*\mathcal{O}_{\PP^{g-1}}(1))$ and $y_1, \ldots, y_{s}$ a basis of $H^0(C,\omega_C \otimes f^*\mathcal{O}_{\PP^{g-1}}(-1))$, where $s=g+1-d$. Consider the $2 \times (g+1-d)$ matrix
 $$ \begin{bmatrix}
    uy_1       &  \dots & uy_{s} \\
    vy_1       &  \dots & vy_{s} \\
   
\end{bmatrix}:=A.$$
The canonical embedding gives an identification $H^0(C,\omega_C) \simeq H^0(\PP^{g-1},\mathcal{O}(1))$ and we may therefore consider $A$ as a matrix of linear forms, that, is as a morphism $$R^{s}(-1) \xrightarrow{A} R^{\oplus 2}.$$ Then $X_f$ is defined be the locus of $2 \times 2$ minors of $A$, or, equivalently, is the locus where $A$ does not have full rank. The homogeneous coordinate ring $\mathcal{O}_{X_f}$ of $X_f$ is then the cokernel of 
$$(\bigwedge^2 R^{s})(-2) \xrightarrow{\wedge^2 A} \bigwedge^2 R^2\simeq R.$$
Since $X_f$ has codimension $g-d$ in $\PP^{g-1}$, $\text{depth}\; I_2(A)=g-d$, so the Eagon--Northcott complex gives a free resolution
\begin{align*}
0 &\leftarrow \mathcal{O}_{X_f} \leftarrow R \xleftarrow{\wedge^2 A} \bigwedge^2 R^{s}(-2) \leftarrow (R^2)^* \otimes \wedge^3 R^s(-3) \leftarrow \text{Sym}_2(R^2)^* \otimes \wedge^4R^s(-4) \\
&\leftarrow \ldots \leftarrow \text{Sym}_{d-2}(R^2)^*\otimes \wedge^sR^s(-s) \leftarrow 0.
\end{align*}
As all differentials in the above resolution are matrices with linear entries, the graded Nakayama lemma immediately implies that this Eagon--Northcott complex is minimal. In particular, the scroll $X_f$ has Betti numbers
$$b_{p,1}(X_f,\mathcal{O}(1))=p {g+1-d \choose p+1},$$
with $b_{p,q}(X_f,\mathcal{O}(1))=0$ for $q \geq 2$. The restriction $\mathcal{O}_{X_f} \twoheadrightarrow \Gamma_C(\omega_C)$ induces \emph{injective} maps
$$\text{Tor}^p(\mathcal{O}_{X_f},\C)_{p+1} \hookrightarrow \text{Tor}^p(\Gamma_C(\omega_C),\C)_{p+1}$$
and hence we derive the bounds $$b_{p,1}(C,\omega_C) \geq p {g+1-d \choose p+1}.$$ In particular, 
$$b_{g-d,1}(C, \omega_C) \geq g-d.$$ Schreyer's Conjecture states that, under appropriate conditions, this is an equality.
\begin{conjecture-nl}[Schreyer]
Suppose $C$ is a curve of gonality $3 \leq d \leq \frac{g+1}{2}$. Assume $C$ has a unique minimal pencil $f: C \to \PP^1$ and that further the Brill--Noether locus $W^1_d(C)=\{ f^*\mathcal{O}(1) \}$ is reduced. Assume furthermore that $f^*\mathcal{O}(1)$ is the unique line bundle achieving the Clifford index. Then
$$b_{g-d,1}(C, \omega_C) = g-d.$$
\end{conjecture-nl}
The condition $d \leq \frac{g+1}{2}$ is precisely that $C$ have non-maximal gonality. 

Any curve satisfying the assumptions and conclusion of Schreyer's Conjecture must also satisfy Green's Conjecture. Indeed, under the assumptions $\text{Cliff}(C)=d-2$ and Green's Conjecture is the statement $b_{d-3,2}(C,\omega_C)=0$. Using Serre duality, this is equivalent to $b_{g-d+1,1}C, \omega_C)=0$. But if $b_{g-d,1}(C, \omega_C) = g-d$, then $\text{Tor}^{g-d}(\mathcal{O}_{X_f},\C)_{g-d+1} \to \text{Tor}^{g-d}(\Gamma_C(\omega_C),\C)_{g-d+1}$ is surjective. But then any linear relation amongst the syzygies in $\text{Tor}^{g-d}(\Gamma_C(\omega_C),\C)_{g-d+1}$ would also be a relation amongst syzygies in $\text{Tor}^{g-d}(\mathcal{O}_{X_f},\C)_{g-d+1}$. As there are no such relations, we must have $b_{g-d+1,1}C, \omega_C)=0$. 

In \cite{schreyer-large}, Schreyer verified his conjecture for general curves of gonality $d$ with $g \gg d$. This has since been verified for general curves of gonality $3 \leq d \leq \frac{g+1}{2}$ in \cite{linear-syz}.

\section{Lattice Polarised K3 Surfaces and the Secant Conjecture}
Consider a projective K3 surface $$X \seq \PP^{g}.$$ If $H \seq \PP^g$ is a hyperplane then the adjunction formula implies $$C:=X \cap H \seq \PP^{g-1} \simeq H$$ is a canonical curve. This simple observation has
deep applications to the study of syzygies of canonical curves. To explain this should be the case, we first introduce some new notation. Let $Y$ be any smooth projective variety and $L,M$ line bundles on $Y$ with $L$ base point free and ample. Set $R:=\text{Sym}\;H^0(Y,L)$ and consider the graded $R$ module 
$$\Gamma_Y(M,L):=\bigoplus_{n \in \mathbb{Z}} H^0(Y,L^{\otimes n} \otimes M).$$ We define $b_{i,j}(Y,M,L):=b_{i,j}(\Gamma_Y(M,L))$.

\begin{prop}[Hyperplane Restriction Theorem, \cite{green-koszul}, \cite{kemeny-farkas}]
Let $X$ be a K3 surface and let $L,H$ be line bundles with $H$ effective and base point free. Assume either (i) $L \simeq \mathcal{O}_X$ or (ii) $(H \cdot L)>0$ and
$H^1(X,qH-L)=0$ for $q \geq 0$. Then for each smooth curve $D \in |H|$
$$b_{p,q}(X,-L,H)=b_{p,q}(D,-L_D,\omega_D)$$
for all $p,q$.
\end{prop}
\begin{proof}
Let $R=\text{Sym} \;H^0(X,H)$. Let $s \in H^0(X,H)$ define $D$. By our assumptions, we have a short exact sequence 
$$0 \to \Gamma_X(-L,H)(-1) \xrightarrow{\otimes s} \Gamma_X(L,H) \to \Gamma_D(L_D,\omega_D) \to 0$$
of $R$ modules. Taking the $(p+q)^{th}$ graded piece of the long exact sequence of $\text{Tor}_R(\; ,\C)$
\begin{align*}
&\to \text{Tor}_R ^p(\Gamma_X(-L,H),\C)_{p+q-1} \xrightarrow{\otimes s} \text{Tor}_R ^p(\Gamma_X(-L,H),\C)_{p+q} \to \text{Tor}_R ^p(\Gamma_D(-L_D,\omega_D),\C)_{p+q}\\
&\to \text{Tor}_R ^{p-1}(\Gamma_X(-L,H),\C)_{p+q-1} \to
\end{align*}
The map $\text{Tor}_R ^p(\Gamma_X(-L,H),\C)_{p+q-1} \xrightarrow{\otimes s} \text{Tor}_R ^p(\Gamma_X(-L,H),\C)_{p+q}$ is zero, \cite[\S 1.6.11]{green-koszul}.
Thus $$\text{Tor}_R ^p(\Gamma_D(-L_D,\omega_D),\C)_{p+q}\simeq \text{Tor}_R ^p(\Gamma_X(-L,H),\C)_{p+q} \oplus \text{Tor}_R ^{p-1}(\Gamma_X(-L,H),\C)_{p+q-1}.$$
Let $R':=\text{Sym} \; H^0(D,\omega_D)$. Using the exact sequence
$$0 \to \C \to H^0(X,H) \to H^0(D,\omega_D) \to 0$$
one sees $$\text{Tor}_R ^p(\Gamma_D(-L_D,\omega_D),\C)_{p+q} \simeq \text{Tor}_{R'} ^p(\Gamma_D(-L_D,\omega_D),\C)_{p+q} \oplus \text{Tor}_{R'} ^{p-1}(\Gamma_D(-L_D,\omega_D),\C)_{p+q-1}.$$
Thus $b_{p,q}(X,-L,H)+b_{p-1,q}(X,-L,H)=b_{p,q}(D,-L_D,\omega_D)+b_{p-1,q}(D,-L_D,\omega_D)$. The claim now follows by induction on $p$, since $b_{-1,q}(X,-L,H)=b_{-1,q}(D,-L_D,\omega_D)=0$.
\end{proof}

The Hyperplane Restriction Theorem is very powerful due to the combination of the following two facts:
\begin{enumerate}
\item Thanks to Voisin's groundbreaking work  as well as the work of Aprodu--Farkas, Green's Conjecture is now known for \emph{any} curve on a K3 surface,  \cite{voisin-even}, \cite{voisin-odd}, \cite{aprodu-farkas-green}.
\item By the Global Torelli Theorem for K3 surfaces, given any even lattice $\Lambda$ of signature $(1,\rho-1)$ with $\rho \leq 10$, there is a nonempty moduli space of K3 surfaces $X$ with $\text{Pic}(X) \simeq \Lambda$, \cite{dolgachev}.
\end{enumerate}
Item $(2)$ gives a very powerful method for constructing examples of curves with prescribed properties. We will illustrate how this works with an application to the Secant Conjecture of Green--Lazarsfeld.

A line bundle $L$ on a curve is said to satisfy property $(N_p)$ if we have the vanishings $$b_{i,j}(C,L)=0 \text{ for} \; i \leq p, \; j \geq 2.$$ In terms of the classical projective geometry, then $\phi_L: C \hookrightarrow \proj^r$ is projectively normal if and only if $L$ satisfies $(N_0)$, whereas the ideal $I_{C/\PP^r}$ is generated by quadrics if, in addition, it satisfies $(N_1)$. The line bundle $L$ is called \emph{$p$-very ample} if for every effective divisor $D$ of degree $p+1$ the evaluation map $$ev: H^0(C,L) \to H^0(D,L_{|_D})$$ is surjective. 

\begin{conjecture}[Secant Conjecture, \cite{green-lazarsfeld-projective}] \label{secant-conj}
Let $L$ be a globally generated line bundle of degree $d$ on a curve $C$ of genus $g$ such that
$$d \geq 2g+p+1-2h^1(C,L)-\text{Cliff}(C).$$
Then $(C,L)$ fails property $(N_p)$ if and only if $L$ is not $p+1$-very ample.
\end{conjecture}
It is rather straightforward to see that if $L$ is not $p+1$ very ample then $b_{p,2}(C,L) \neq 0$. The harder part is to go in the other direction. 

In the case $h^1(C,L) \neq 0$, then the Secant Conjecture reduces to Green's Conjecture, so we will focus on the case of a \emph{non-special} line bundle $L$, i.e.\ one with $h^1(C,L)=0$.
\begin{thm}[\cite{kemeny-farkas}] \label{generic-secant}
The Secant Conjecture holds for a general curve $C$ of genus $g$ and a general line bundle $L$ of degree $d$ on $C$.
\end{thm}
An elementary argument shows that if $C$ is general then the general $L \in \text{Pic}^d(C)$ is $(p+1)$-very ample if and only if
$$d \geq g+2p+3.$$
Using this inequality and the fact that if $L$ is a globally generated, nonspecial, line bundle with $b_{p,2}(C,L)=0$ then $b_{p-1,2}(C,L(-x))=0$ for a general $x \in C$, Theorem \ref{generic-secant} reduces to finding a general curve $C$ together with a non-special line bundle $L \in \text{Pic}^d(C)$ with $b_{p,2}(C,L)=0$ in the following two cases
\begin{enumerate}
\item $g=2i+1$, $d=2p+2i+4$, $p \geq i-1$
\item $g=2i$, $d=2p+2i+3$, $p \geq i-1$.
\end{enumerate}

We construct such curves $C$ and line bundles $L$ using lattice polarized K3 surfaces. Consider first the odd genus case $g=2i+1$. Let $\Lambda=\mathbb{Z}[C]\oplus \mathbb{Z}[L]$ be a lattice with intersection pairing
$$\begin{pmatrix}
(C)^2 & (C \cdot L)\\
(C \cdot L) & (L)^2
\end{pmatrix}=\begin{pmatrix}
4i & 2p+2i+4\\
2p+2i+4 & 4p+4
\end{pmatrix}.$$
Consider a general K3 surface $X$ with $\text{Pic}X \simeq \Lambda$. We need to prove that $b_{p,2}(C,L_{C})=0$. From the short exact sequence
$$0 \to \Gamma_X(-C,L) \to \Gamma_X(L) \to \Gamma_C(L_{C}) \to 0$$
we get
$$ \to \text{Tor}^p_R(\Gamma_X(L),\C)_{p+2} \to \text{Tor}^p_R(\Gamma_C(L_C),\C)_{p+2} \to \text{Tor}^{p-1}_R(\Gamma_X(-C,L),\C)_{p+2} \to $$
for $R= \text{Sym}\; H^0(X,L) \simeq \text{Sym}\; H^0(C,L_C).$ So it suffices to prove $b_{p,2}(X,L)= b_{p-1,3}(X,-C,L)=0$, or, by the Hyperplane Restriction Theorem
$$b_{p,2}(D,\omega_D)= b_{p-1,3}(D,\mathcal{O}_D(-C),\omega_D)=0$$
where $D \in |L|$ is a smooth curve. 

An easy computation shows $D$ has genus $2p+3$ and Clifford index $p+1$, so the vanishing $b_{p,2}(D,\omega_D)=0$ follows from Green's Conjecture. For the vanishing $b_{p-1,3}(D,\mathcal{O}_D(-C),\omega_D)=0$, we need to replace the use of Green's Conjecture with the following important result:
\begin{thm}[\cite{farkas-mustata-popa}]
Let $C$ be a non hyperelliptic curve of genus $g$ and $\eta \in \text{Pic}^{g-2j-1}(C)$ a line bundle such that $\eta \notin C_{g-j-1}-C_j$ for some $1 \leq j \leq \frac{g-1}{2}$. Then $$b_{j-1,3}(C,-(\omega_C\otimes\eta),\omega_C)=0.$$
\end{thm}
This implies that that for a general line bundle $\eta$ of degree $\leq 2$, $b_{p-1,3}(D,-(\omega_D \otimes \eta),\omega_D)=0$, where the above result further specifies what is meant by ``general''. We apply this in our case to the line bundle $L^*_D(C)$ which has degree $2i-2p \leq 2$.

The case of even genus $g=2i$ is similar. In this case one uses K3 surfaces with Picard lattice
$$\begin{pmatrix}
4i-2 & 2p+2i+3\\
2p+2i+3 & 4p+4
\end{pmatrix}.$$

\section{The Wahl Map and the Prym--Green Conjecture}
Consider the moduli space $\mathcal{R}_{g,\ell}$ parametrising pairs $(C,\tau)$ of a smooth, genus $g$ curve together with a torsion bundle $\tau$ of order exactly $\ell$. This is an irreducible moduli space which admits a compactification $\overline{\mathcal{R}}_{g,\ell}$, \cite{chiodo-stable}, \cite{chiodo-farkas}. Here are a few reasons one might be interested in 
$\mathcal{R}_{g,\ell}$:
\begin{enumerate}
\item $\mathcal{R}_{g,\ell}$ can be considered as a higher genus analogue of the modular curve parametrizing elliptic curves plus $\ell$-torsion line bundles and as such ought to be interesting from a number-theoretic point of view.
\item The moduli space $\mathcal{R}_{g,\ell}$ is very closely related to the stack of $\ell$-spin curves $\{ (C,L) \; | \; L^{\otimes \ell} \simeq \omega_C \}$ and the two spaces are often considered together, \cite{chiodo-stable}. The space of $\ell$-spin curves has important applications to Gromov--Witten theory, \cite{witten-matrix}.
\item In the case $\ell=2$, the space $\mathcal{R}_{g,2}$ has been much studied in relation to Abelian varieties, due to a construction of Prym which associates an Abelian variety to a point $(C,\tau) \in \mathcal{R}_{g,2}$, \cite{mumford-prym}.
\end{enumerate}
Let us explain $(3)$ in more detail. There is the \emph{Prym map} $$P_g \; : \; \mathcal{R}_{g,2} \to \mathcal{A}_{g-1}$$ defined as follows. Let $(C,\tau) \in \mathcal{R}_{g,2}$ be a point and consider the associated double cover
$$\nu \; : \; \widetilde{C} \to C$$
which has the property $\nu_*\mathcal{O}_{\widetilde{C}} \simeq \mathcal{O}_C \oplus \tau$. Pushforward of divisors defines the \emph{Norm map}
$$\text{Nm}_{\nu} \; : \; \text{Pic}^{2g-2}(\widetilde{C}) \to \text{Pic}^{2g-2}(C).$$
Then $\text{Nm}_{\nu}^{-1}(\omega_C)$ has two isomorphic connected components. The Abelian variety $P_g(C,\tau)$ is the component
$$\{ L \in \text{Nm}_{\nu}^{-1}(\omega_C) \; : \; h^0(C,L)=0 \; \; \text{mod $2$} \}$$
with principal polarization given by the Theta divisor
$$\Theta := \{ L \in \text{Nm}_{\nu}^{-1}(\omega_C) \; : \; h^0(C,L)\geq 2,\;  h^0(C,L)=0 \; \; \text{mod $2$} \}$$

If $(C,\tau) \in \mathcal{R}_{g,\ell}$, then the associated \emph{paracanonical curve} is the embedded curve
$$\phi_{\omega_C \otimes \tau} \; : \; C \hookrightarrow \PP^{g-2}.$$
In the case $\ell=2$, Mumford noticed that there was a close relationship between the projective geometry of a general paracanonical curve and the geometry of the Prym map. Indeed, the differential of the Prym map at a point $(C,\tau) \in \mathcal{R}_{g,2}$
$$d P_g \; : \; H^0(C,\omega_C^{\otimes 2})^* \to (\text{Sym}^2 \; H^0(C,\omega_C \otimes \tau))^*$$
is injective if and only if the multiplication map 
$$\text{Sym}^2 \; H^0(C,\omega_C \otimes \tau) \to H^0(C,\omega_C^{\otimes 2})$$
is surjective, i.e.\ if and only if the corresponding paracanonical curve is projectively normal. Using a degeneration argument, Beauville verified that this indeed holds for the general $(C,\tau) \in \mathcal{R}_{g,2}$, provided $g \geq 6$, \cite{beauville-prym}. Debarre went one step further and showed that the ideal $I_{C/{\PP^{g-2}}}$ is generated by quadrics for $(C,\tau) \in \mathcal{R}_{g,2}$ general and $g \geq 9$, \cite{debarre-torelli}. Using this, he was able to conclude that $P_g$ is in fact generically injective for $g \geq 9$.

It is tempting to imitate Green's conjecture and try to generalize the Beauville--Debarre results to higher syzygies. This is achieved in the following result, which answers affirmatively a conjecture of Farkas--Ludwig \cite{farkas-ludwig} and Chiodo--Eisenbud--Farkas--Schreyer, \cite{chiodo-eisenbud-farkas-schreyer}:
\begin{thm}[\cite{pgodd-all}] \label{pgodd-all}
Let $g=2i+5$ be odd and $(C,\tau) \in \mathcal{R}_{g,\ell}$ general. Then $$b_{i+1,1}(C,\omega_C\otimes \tau)=b_{i-1,2}(C,\omega_C\otimes \tau)=0.$$
\end{thm}
The result above suffices to completely determine \emph{all} Betti numbers $b_{p,q}(C,\omega_C\otimes \tau)$ of a general paracanonical curve of arbitrary level $\ell$ and odd genus $g=2i+5$. Indeed, the \emph{Betti table}, i.e.\ table with $(q,p)^{th}$ entry $b_{p,q} (C,\omega_C\otimes \tau)$ of such a curve is
\begin{table}[htp!]
\begin{center}
\begin{tabular}{|c|c|c|c|c|c|c|c|c|}
\hline
$1$ & $2$ & $\ldots$ & $i-1$ & $i$ & $i+1$ & $i+2$  & $\ldots$ & $2i+2$\\
\hline
$b_{1,1}$ & $b_{2,1}$ & $\ldots$ & $b_{i-1,1}$ & $b_{i,1}$ & 0 & 0 &  $\ldots$ & 0 \\
\hline
$0$ &  $0$ & $\ldots$ & $0$ & $b_{i,2}$ & $b_{i+1,2}$ & $b_{i+2,2}$ & $\ldots$ & $b_{2i+2,2}$\\
\hline
\end{tabular}
\end{center}
 
\end{table}

 $$\text{where \;} b_{p,1}=\frac{p(2i-2p+1)}{2i+3}{2i+4\choose p+1} \ \  \mbox{ if } p\leq i\  \mbox{,} \ \  b_{p,2}=\frac{(p+1)(2p-2i+1)}{2i+3}{2i+4\choose p+2} \  \ \  \mbox{ if } p\geq i.$$
In even genus, it is known that the analogous conjecture does not always hold see \cite{chiodo-eisenbud-farkas-schreyer}, \cite{CFVV}.

In order to prove Theorem \ref{pgodd-all}, one might like to use the method described in the previous section and work with suitable K3 surfaces. This was successfully carried out assuming that the torsion level is large compared to the genus, \cite{farkas-kemeny-high}. When one attempts to prove the result for \emph{all} levels, one immediate difficulty arises. Since the Picard group of a K3 surface is torsion free, it is difficult to construct K3 surfaces containing paracanonical curves in such a way that the torsion bundle $\tau$ comes as a pull-back of a line bundle on the surface, at least if $\ell$ is arbitrary. The solution we take in \cite{pgodd-all} is to look for something which is similar to a K3 surface but for which the Picard group admits torsion. 

To explain how such surfaces come about naturally, we need to describe the work of Wahl on classifying curves lying on a K3 surface, \cite{wahl-square}. Let $C$ be a smooth, complex curve. The \emph{Wahl map}
$$\bigwedge^2H^0(C,\omega_C) \to H^0(C,\omega_C^{\otimes 3}),$$
is defined by the following rule. Choose analytic coordinate charts for $C$. A section $s \wedge t$ is mapped under the Wahl map to the section specified by the formula $(ds)t-t(ds)$ on these charts, see also \cite{wahl-gaussian}.

Let $C \seq \PP^{g-1}$ be a canonical curve and let $$Y \seq \PP^{g},$$
denote the cone over $C$. Wahl discovered his map by studying the graded module of first order deformations of $Y$, \cite{wahl-jac}. The first interesting graded piece of this module is precisely the cokernel of the Wahl map. In particular, if the Wahl map is nonsurjective, and if furthermore a certain obstruction group vanishes, then $Y$ may be deformed to a Gorenstein surface $$X' \seq \PP^{g-1}$$
with $\omega_{X'} \simeq \mathcal{O}_{X'}$ which is not a cone. Such a surface looks somewhat like a K3 surface, with the caveat that it may have nasty singularities. This led Wahl to conjecture:
\begin{conjecture-nl}[Wahl's Conjecture]
Let $C$ be a curve which is Brill--Noether--Petri general. Then $C$ lies on a K3 surface if and only if the Wahl map is nonsurjective.
\end{conjecture-nl}
One direction of this is relatively easy: if $C$ lies on a K3 surface then it follows from the infinitesimal analysis above that the Wahl map is never surjective, see \cite{wahl-jac}. Also see \cite{beauville-merindol} for a simple, direct proof of this fact without deformation theory.

As stated above, Wahl's conjecture needs a slight modification. Arbarello--Bruno--Sernesi proved the following:
\begin{thm-nl}[\cite{arbarello-bruno-sernesi}]
Let $C$ be a Brill--Noether--Petri general curve of genus $g \geq 12$. Then $C$ lies on a projective K3 surface $X \seq \PP^g$, or is a limit of such curves, if and only if the Wahl map is nonsurjective.
\end{thm-nl}
The result above is shown to be optimal in \cite{arbarello-bruno-halphen}. The proof of Arbarello--Bruno--Sernesi proceeds as follows. First of all, they verify that the obstruction group of \cite{wahl-square} vanishes for a Brill--Noether--Petri general curve $C$. Thus if the Wahl map of $C$ is nonsurjective, the cone $Y$ can be deformed to some Gorenstein surface $X' \neq Y$ with canonical curves as hyperplane sections. Such surfaces were classified by Epema in his thesis, \cite{epema}. This boils the task down to deciding when the surfaces appearing in Epema's list can be smoothed (the smoothing of any such surface must be a K3 surface).

In particular, there is one class of surface which features prominently in \cite{arbarello-bruno-sernesi}. These are surfaces whose desingularization is a projective bundle over an elliptic curve. Such surfaces are an excellent candidate for proving the Prym--Green conjecture for the following reasons: $(i)$ they arise as degenerations of smooth K3 surfaces, $(ii)$ their general hyperplane sections are Brill--Noether--Petri general \cite{farkas-tarasca}, $(iii)$ by pulling back bundles from the elliptic curve, we have an abundance of torsion line bundles to work with. 

More precisely, let $E$ be an elliptic curve and set $$X:=\PP(\mathcal{O}_E \oplus \eta)$$ where $\eta \in \text{Pic}^0(E)$ is neither trivial nor torsion. Then $$\phi \; : \; X \to E$$
admits two sections $J_0$ and $J_1$ corresponding to the quotients $\eta$ and $\mathcal{O}_E$ respectively. Let $r \in E$ be general, set $f_r:=\phi^{-1}(r)$ and consider the linear system $|g J_0+f_r|$. The general curve $C \in |g J_0+f_r|$ is smooth of genus $g$ and passes through two base points $x \in J_0$ and $y \in J_1$. If $\widetilde{X}$ denotes the blow-up at these two points, then $$K_{\widetilde{X}} \simeq -(J'_0+J'_1),$$
where $J'_0,J'_1$ are the proper transforms of $J_0,J_1$. The surface $\widetilde{X}$ is the resolution of a limiting K3 surface $Y \seq \PP^g$ with two elliptic singularities.

Now let $b \in E$ be such that $\tau=b-r$ is $\ell$-torsion, write $\eta=a-b$ for some $a \in E$ and set $L=(g-2)J_0+f_a$ on $X$. By adjunction
$$L_C \simeq K_C+\tau$$
and we prove that the paracanonical curve $(C,L_C)$ satisfies the Prym--Green conjecture. One of the crucial new ingredients in the proof is that there is a canonical degeneration of $C$ to $J_0 \cup D$ where $D \in |(g-1) J_0+f_r|$ has genus $g-1$, which allows one to make inductive arguments on the genus.

\section{Divisors on Moduli Spaces and the extremal Betti number of a canonical curve}
It has been known for some time that some of the most interesting loci in the moduli space of curves can be constructed using syzygies. For example,
consider the divisor $$\mathcal{K} := \{ C \in \mathcal{M}_{10} \; | \; \exists L \in \text{Pic}^{12}(C) \; \text{with $h^0(L)=4$ and $b_{0,2}(C,L) \neq 0$} \}$$
in the moduli space of curves of genus $10$. It was shown in \cite{farkas-popa} that the closure of this locus violates the famous Slope Conjecture of Harris--Morrison. In practice, such syzygetic loci tend to give more information about the birational geometry of moduli spaces than other kinds of loci (such as Brill--Noether loci), \cite{farkas-koszul}, \cite{chiodo-eisenbud-farkas-schreyer}.

Rather than using our knowledge of syzygies of curves to describe the geometry of moduli spaces, we can also reverse the process and use cycle calculations on moduli spaces in order to obtain information about syzygies of curves. In \cite{hirschowitz-ramanan}, Hirschowitz--Ramanan construct determinantally a divisor $\mathcal{K}os \seq \mathcal{M}_{2k-1}$ parametrising
$$\{ C \in \mathcal{M}_{2k-1} \; | \; b_{k-1,1}(C,\omega_C) \neq 0 \}$$
and show that it coincides set-theoretically with the divisor $\mathcal{H}ur \seq \mathcal{M}_{2k-1}$ parametrising
$$\{ C \in \mathcal{M}_{2k-1} \; | \; \text{gon}(C) \leq k \}$$
studied by Harris--Mumford, \cite{harris-mumford}. Further, as divisors
$$\mathcal{K}os=(k-1)\mathcal{H}ur \; \in \;  A^1(\mathcal{M}_{2k-1},\mathbb{Q}).$$
Using this, Hirschowitz and Ramanan concluded that Green's Conjecture holds for any curve of genus $g=2k-1$ and, furthermore, Schreyer's Conjecture holds for curves of genus $g=2k-1$ and gonality $k$.

It was discovered by Aprodu that one can use the Hirschowitz--Ramanan computation to obtain results about curves of \emph{arbitrary} gonality, \cite{aprodu-remarks}. Let $C$ be a curve of genus $g$ and non-maximal gonality $k \leq \lfloor \frac{g}{2} \rfloor+2$. Let $W^1_m(C)$ denote the subvariety of $\text{Pic}^m(C)$ consisting of line bundles with at least two sections. We say $C$ satisfies \emph{linear growth} if
$$\dim W^1_{k+n}(C) \leq n \text{  for all $ 0 \leq n \leq g-2k+2$}.$$
Aprodu proved that the general $k$-gonal curve satisfies linear growth (if $k$ is non-maximal) and further:
\begin{thm-nl} [Aprodu]
Let $C$ satisfy linear growth. Then $C$ satisfies Green's Conjecture.
\end{thm-nl}
Aprodu's Theorem is the sharpest known result on Green's conjecture. Further, it was a key step in verifying that Green's Conjecture holds for curves on arbitrary K3 surfaces,
\cite{aprodu-farkas-green}. 

Aprodu's Theorem relies on the following trick with nodal curves (which is a variant of an argument of Voisin \cite{voisin-even}): let $C$ and be in the theorem, with $k=\text{gon}(C)$, choose $n=g+3-2k$ general pairs of points $(x_i,y_i) \in C$, $1 \leq i \leq n$ and let $D$ be the nodal curve obtained from $D$ by identifying $x_i$ and $y_i$. Then $D$ has genus $g+n=2(g-k+1)-1$ and Aprodu shows that the linear growth condition implies $D \notin \overline{\mathcal{H}ur}$. By Hirschowitz--Ramanan's calculation, this implies $b_{g-k+1,1}(D,\omega_D)=0$. But, as Voisin shows, there is a natural injective map
$$\text{Tor}^{g-k+1}_{R_1}(C,\omega_C)_{g-k+2} \hookrightarrow \text{Tor}^{g-k+1}_{R_2}(D,\omega_D)_{g-k+2},$$
where $R_1=\text{Sym} H^0(\omega_C)$, $R_2=\text{Sym} H^0(\omega_D)$. Hence $b_{g-k+1,1}(C,\omega_C)=0$ as predicted by Green's Conjecture.

We now turn back to the problem of attempting to describe the Betti table of a canonical curve of non-maximal gonality $k$. Recall from Section \ref{Schreyer} that if $C$ admits $m$ minimal pencils $f_1, \ldots, f_m$ then the associated scrolls $X_{f_1}, \ldots, X_{f_m}$ each contribute to the syzygies of $C$. Recall further that the ``extremal'' Betti number $b_{g-k,1}(C,\omega_C)$ was singled out by Schreyer's conjecture to be of particular importance.
We prove
\begin{thm} [\cite{in-prep}]
Let $C$ be a smooth curve of genus $g$ and gonality $k \leq \lfloor \frac{g+1}{2} \rfloor$. Assume $C$ admits $m$ minimal pencils, and that the pencils are infinitesimally and geometrically in general position.  Assume further that $C$ satisifies bpf-linear growth. Then 
$$b_{g-k,1}(C,K_C)=m(g-k).$$ 
\end{thm}
In other words, one can read off the number of minimal pencils from the last Betti number in the linear strand, under certain generality assumptions on the minimal pencils. Let us first state the meaning of the assumptions. A curve $C$ of genus $g$ and gonality $k$ satisfies \emph{bpf-linear growth} provided we have the dimension estimates
\begin{align*}
\dim W^1_{k+m}(C) &\leq m,  \;\; \text{for $0 \leq m \leq g-2k+1$} \\
\dim W^{1,\mathrm{bpf}}_{k+m}(C) &<m, \;\; \text{for $0 < m \leq g-2k+1$}, \end{align*} 
where $W^{1,\mathrm{bpf}}_{k+m}(C)$ denotes the locus of base point free line bundles. This condition is a slight strengthening of Aprodu's linear growth assumption. Next, the condition that the pencils are infinitesimally in general position means that the deformation theory of any subset $\sigma=\{ \sigma_i \} \seq \{f_1, \ldots, f_m \}$ of the pencils is unobstructed (modulo the $\text{PGL}(2)$ action). More precisely, setting $F_{\sigma}:=(f_{\sigma_i}) \; : \; C \to (\PP^1)^{|\sigma|}$, we require
$$\text{Ext}^2_C(\Omega^{\bullet}_{F_{\sigma}},\mathcal{O}_C(-p-q-r))=0$$
for all subsets $\sigma$ and general $p,q,r \in C$, where $\Omega^{\bullet}_{F_{\sigma}}$ is the cotangent complex of $F_{\sigma}$. 

Lastly, the condition that the pencils be geometrically in general position is a condition to ensure that the scrolls contribute syzygies independently into the extremal Betti number of the curve. To describe it, choose a general divisor $T$ of degree $g-1-k$ on $C$. Let $Q_{f_1}, \ldots, Q_{f_m}$ be the quadrics obtained by projecting the scrolls $X_{f_1}, \ldots, X_{f_m}$  away from $T$. We say $\{ f_1, \ldots, f_m \}$ is geometrically in general position if the set $$\{ Q_{f_1}, \ldots, Q_{f_m} \} \seq |\mathcal{O}_{\PP^k}(2)|$$ is in general position. 

In practice, these three assumptions seem relatively easy to verify. For instance, they can be checked to hold for a general curve of non-maximal gonality $k$ admitting $m \leq 2$ minimal pencils. When $m \geq 3$, we lack a good understanding of when the moduli space of curves with $m$ minimal pencils is nonempty and irreducible, which makes it harder to approach this case, but computational evidence suggests, for instance, that the assumptions are verified for $g=11$ and $k=6$ provided $1 \leq m \leq 10$, which excludes only the ``sporadic" cases $m=12, 20$ (the assumptions should be satisfied up until the first $m$ where the moduli space of curves with $m$ pencils becomes empty, which in this case is $m=11$).

We just state a few words about the proof, for more details see the forthcoming \cite{in-prep}. The proof works by ultimately reducing to the case $g=2k-1$ using a variant of Aprodu's trick. The reduction is significantly more difficult than in Aprodu's case, as we are required to work with stable maps with unstable curves on the base, and involves the notion of ``twisting'' for line bundles on a family of curves with central fibre a reducible curve, see \cite{farkas-pand}, \cite{BCGGM}. 

In the case $g=2k-1$, in our setting the role of the Koszul divisor of Hirschowitz--Ramanan is replaced with ``Eagon--Northcott" divisors defined on an appropriate moduli space $\mathcal{H}(m)$ of stable maps $C\to (\PP^1)^m$  with three base points. Letting $\mathcal{H}(m) \to \mathcal{H}(1)$, denote the projection to the first factor, these divisors push forward to give codimension $m+1$ cycles
$$\mathcal{EN}_m \in A^{m+1}(\mathcal{H}(1),\mathbb{Q}),$$
where $\H(1)$ is a suitable subset of the space of degree $k$ stable maps $C \to \PP^1$ from curves of genus $2k-1$ (with three base points, to account for the automorphisms of $\PP^1$). These cycles satisfy the identity $$\mathcal{EN}_{m}=(k-1)\mathcal{BN}_{m+1}$$
where $\mathcal{BN}_{m+1} \in A^m(\H(1), \mathbb{Q})$ is a cycle corresponding to curves with $m+1$ minimal pencils. This equation should be seen as a generalization of Hirschowitz--Ramanan's equation 
$$\mathcal{K}os=(k-1)\mathcal{H}ur  \; \in \;  A^1(\mathcal{M}_{2k-1},\mathbb{Q}).$$

\bibliography{biblio}
\end{document}